\documentclass[ECP,preprint]{ejpecp}
\usepackage{xcolor}

\newcommand{\R}{\mathbb{R}}
\newcommand{\N}{\mathbb{N}}

\newcommand{\pP}{\mathbb{P}}
\newcommand{\bP}[1]{\pP\bb{#1}}
\newcommand{\E}{\mathbb{E}}
\newcommand{\CE}[2]{\E\Cb{#1\,\middle|\,#2}}
\newcommand{\CP}[2]{\pP\bb{#1\,\middle|\,#2}}

\newcommand{\ind}{\mathbf{1}}

\newcommand{\bb}[1]{\left(#1\right)}
\newcommand{\Cb}[1]{\left[#1\right]}
\newcommand{\cb}[1]{\left\{#1\right\}}
\newcommand{\norm}[1]{\left\Vert#1\right\Vert}
\newcommand{\abs}[1]{\left|#1\right|}

\newcommand{\Lcal}{\mathcal{L}}
\newcommand{\Dcal}{\mathcal{D}}

\newcommand{\defeq}{\mathrel{\mathop:}=}

\SHORTTITLE{Branching Brownian motion in a self-generated resource environment}
\TITLE{A sharp survival threshold for branching Brownian motion in a self-generated resource environment}
\AUTHORS{
  Maciej~G{\l}uchowski\footnote{University of California, Los Angeles, USA. \EMAIL{mgluchowski@math.ucla.edu}}\orcid{0009-0009-4239-9674}%
  \and
  Georg~Menz\footnote{University of California, Los Angeles, USA. Corresponding author. \EMAIL{gmenz@math.ucla.edu}}\orcid{0000-0001-9295-5063}}

\AMSSUBJ{60K35}
\AMSSUBJSECONDARY{60J80}

\hypersetup{
  pdftitle={A sharp survival threshold for branching Brownian motion in a self-generated resource environment},
  pdfauthor={Maciej G{\l}uchowski and Georg Menz}
}

\ABSTRACT{
We study branching Brownian motion in a resource environment created by the particles themselves. Particles reproduce at rate~$\beta$ times the resource nearby, resource is consumed at rate~$c$ times the number of particles nearby, and each death returns resource mass~$m$. Let~$\theta=\beta m/c$, which represents the number of births supported by resource from one corpse, assuming perfect efficiency. We prove that the population dies out almost surely if~$\theta\leq1$, even with perfect recycling ($\theta=1$). If~$\theta>1$ and the interaction kernel has finite range and is bounded below at short distances, the population survives with positive probability from every initial state with at least two particles, or with one particle and some resource. The threshold does not depend on the diffusivity, the death rate or the interaction kernel.
}
\begin{document}

\section{Introduction}

Branching Brownian motion is a basic model of a population that moves, reproduces and dies: particles move as independent Brownian motions and give birth and die at constant rates. Starting from finitely many particles, the population survives with positive probability exactly when the birth rate exceeds the death rate. For branching random walks in a random environment, survival is decided by the environment, with different criteria for local and global survival~\cite{GantertMullerPopovVachkovskaia2010}. The environment can also be consumed by the particles, as in cookie branching random walks, where particles branch and move differently at sites carrying a cookie and a cookie is removed once its site is visited~\cite{BartschEtAl2013}. In this paper the environment is not given in advance. It is a resource that the particles consume, and that they themselves create when they die.

More precisely, we study Brownian particles with diffusivity~$\sigma^2$, with positions represented by a counting measure~$\eta_t$, coupled to a resource measure~$\mu_t$. A particle gives birth at a rate proportional to the resource around it, and resource is consumed at a rate proportional to the number of particles around it. Particles die at a constant rate~$\delta$ and then return resource mass~$m>0$ at the death location. There is no external resource supply. We ask when this process survives with positive probability. The question is motivated by microbial populations without external food supply: experiments and models of bacterial starvation show that recycling material from dead cells can help sustain a population~\cite{TakanoEtAl2017,SchinkEtAl2019}. In spatial population models, resource limitation is usually modeled through crowding, which increases death rates~\cite{MaillardPenington2024} or reduces birth rates~\cite{BirknerDepperschmidt2007}; survival and extinction of such locally regulated populations have been studied extensively~\cite{Etheridge2004,BlathEtheridgeMeredith2007,BirknerDepperschmidt2007,BirknerEtAl2024,MaillardPenington2024}; see~\cite{FournierMeleard2004} for individual-based models. Another approach includes resources explicitly, as in stochastic chemostat models~\cite{ColletEtAl2013} and spatial simulations with externally supplied resources~\cite{WelkerStark2024,Kochocki2026}. To our knowledge, a spatial model in which resource is created only by recycling has not been analyzed rigorously before.

Let~$\beta>0$ and~$c>0$ be the birth and consumption coefficients, and set~$\theta\defeq\frac{\beta m}{c}$. This parameter has a direct interpretation: one corpse, if completely consumed, supports~$\theta$ births on average. If~$\theta<1$, resource is lost in every life cycle,~$\theta=1$ corresponds to perfect recycling, and~$\theta>1$ to a net gain. In a well-mixed, mean-field model the survival threshold is therefore~$\theta=1$. In our spatial model one would expect a larger threshold: a corpse is fully consumed only if particles stay near it, but particles move away and die, and resource left behind is wasted. The threshold might moreover depend on the diffusivity, the death rate and the interaction kernel, just as spatial correlations push the critical infection rate of the contact process above its mean-field value~\cite[Chapter~VI]{Liggett1985}.

Our main result, Theorem~\ref{thm:main}, shows that this does not happen: the threshold is exactly~$\theta=1$. If~$\theta\leq1$, the process dies out almost surely from every finite initial state. If~$\theta>1$ and particles interact with resource only within a fixed distance, with interaction strength bounded below at short distances, then survival has positive probability exactly when the initial state contains at least two particles, or one particle and some resource. The threshold does not depend on the dimension, the diffusivity, the death rate, the speed of consumption or the shape of the interaction kernel. This is in contrast to branching systems in a given environment~\cite{GantertMullerPopovVachkovskaia2010} and to locally regulated populations, where survival depends on dispersal relative to the interaction range~\cite{Etheridge2004}.

For~$\theta\leq1$, the total mass~$\mu_t\bb{\R^d}+\frac{c}{\beta}\eta_t\bb{\R^d}$ is a nonnegative supermartingale and therefore converges. Each birth increases it by~$\frac{c}{\beta}$, so only finitely many births occur and the population dies out. This direction is elementary. The case~$\theta>1$ is harder because the threshold is sharp: for~$\theta$ close to~$1$, at most a fraction~$1-1/\theta$ of each corpse may be wasted, so an argument that loses a fixed fraction of the resource proves survival only for large~$\theta$. Moreover, unlike in classical branching Brownian motion, particles do not branch independently: each birth rate depends on the resource that particles share. We show that dense colonies consume their corpses with little waste and that this consumption necessarily produces new births; the latter is Lemma~\ref{lem:short-time-growth-new}, the main technical part of the proof. Since~$\theta>1$, these births more than make up for deaths, which allows us to carry out a block argument for colony growth.

Like every mathematical model, ours is an idealization. It retains individual life spans, local consumption and reproduction, motion of the particles, and the return of resource at death. It neglects, among other effects, motion and decay of the resource, delays between death and the release of resource, crowding, saturation of consumption, and dependence of the death rate on starvation. The death rate~$\delta$ models background mortality, for instance from radiation, freezing, desiccation or predation. The model thus describes the population over a time scale that is long compared with a single division but short compared with the time scale on which the neglected effects matter. If the resource is conserved, producing a particle requires at least the resource it returns at death, which forces~$\theta\leq1$. The regime~$\theta>1$ therefore describes a population in which particles have access to ambient building blocks or energy, like minerals or light, but lack organic material which they must create and recycle. An example is the colonization of barren mineral substrates such as fresh volcanic deposits or glacier forefields, where organic carbon is scarce and pioneer microbes live in part on atmospheric trace gases or ancient carbon~\cite{BardgettEtAl2007,KingWeber2008}. Pioneer communities also include photosynthetic and nitrogen-fixing bacteria that add organic carbon and nitrogen~\cite{SchmidtEtAl2008}, so a corpse can return more organic material than its production consumed. A more speculative example is a habitable but uninhabited environment such as Mars~\cite{CockellEtAl2012}. In the early stages of colonization crowding is negligible. We have kept the model minimal, so that the threshold is exact and its proof is short.
Some of the neglected effects should not move the threshold: decay or diffusion of the
resource act on a corpse only until it is consumed, and a dense colony consumes all nearby resource almost immediately. Other effects change the mechanism. A death rate that
increases under starvation couples mortality to the resource, and crowding introduces
the local regulation studied in~\cite{BirknerDepperschmidt2007,Etheridge2004,BirknerEtAl2024};
whether a sharp threshold persists in these settings is open. Quantitative aspects of survival are not addressed here.

\section{The model}\label{sec:model}

We define the process through four mechanisms: diffusion, birth, resource consumption, and death with return of resource. For~$d\geq1$, the state space consists of pairs~$\bb{\eta,\mu}$, where~$\eta$ is a finite counting measure on~$\R^d$ and~$\mu$ is a finite nonnegative Borel measure on~$\R^d$. For a finite measure~$\nu$ and bounded measurable~$g$, write~$\langle\nu,g\rangle\defeq\int_{\R^d}g\,d\nu$. Let~$K\colon\R^d\to[0,\infty)$ be bounded, symmetric, Borel, and not identically zero, and for every finite nonnegative measure~$\lambda$ set the convolution~$\bb{K*\lambda}\bb{x}
=
\int_{\R^d}K\bb{x-y}\,\lambda(dy)$.

Fix~$\beta,\delta,c,m,\sigma>0$. A particle gives birth with coefficient~$\beta$, dies at rate~$\delta$, and returns resource mass~$m$ at death. Resource is consumed with coefficient~$c$, and particles move as independent Brownian motions with diffusivity~$\sigma^2$. For cylinder functions
\begin{equation*}
F\bb{\eta,\mu}
=
f\bb{
\langle\eta,\varphi_1\rangle,\dots,\langle\eta,\varphi_n\rangle,
\langle\mu,\psi_1\rangle,\dots,\langle\mu,\psi_k\rangle
},
\qquad
f\in C_b^2\bb{\R^{n+k}},
\end{equation*}
with~$\varphi_i\in C_b^2\bb{\R^d}$ and~$\psi_j$ bounded measurable, the generator decomposes as
\begin{equation*}
\Lcal
=
\Lcal_{\mathrm{diff}}
+
\Lcal_{\mathrm{birth}}
+
\Lcal_{\mathrm{con}}
+
\Lcal_{\mathrm{death}}.
\end{equation*}
All derivatives of~$f$ below are evaluated at~$\bb{
\langle\eta,\varphi_1\rangle,\dots,\langle\eta,\varphi_n\rangle,
\langle\mu,\psi_1\rangle,\dots,\langle\mu,\psi_k\rangle
}$.
\paragraph{Diffusion.}
Particles move as independent Brownian motions with generator~$\frac{\sigma^2}{2}\Delta$. The resource does not diffuse. Writing~$\eta=\sum_{l=1}^N\delta_{x_l}$ we have
\begin{equation*}
\Lcal_{\mathrm{diff}}F\bb{\eta,\mu}
=
\frac{\sigma^2}{2}\sum_{l=1}^N\Delta_{x_l}F\bb{\eta,\mu}.
\end{equation*}

\paragraph{Births.}
Each particle at~$x$ gives birth at its current location at rate~$\beta\bb{K*\mu}\bb{x}$:
\begin{equation*}
\Lcal_{\mathrm{birth}}F\bb{\eta,\mu}
=
\beta
\int_{\R^d}
\bb{K*\mu}\bb{x}
\Cb{
F\bb{\eta+\delta_x,\mu}-F\bb{\eta,\mu}
}
\,\eta(dx).
\end{equation*}

\paragraph{Resource consumption.}
Resource at~$y$ is consumed at rate~$c\bb{K*\eta}\bb{y}$:
\begin{equation*}
\Lcal_{\mathrm{con}}F\bb{\eta,\mu}
=
-c\sum_{j=1}^k\partial_{n+j}f
\int_{\R^d}
\psi_j\bb{y}\bb{K*\eta}\bb{y}\,\mu(dy).
\end{equation*}
Equivalently, if no deaths occur on~$(s,t]$, then the resource evolves deterministically:
\begin{equation}\label{eq:resource-between-deaths}
\mu_t(dy)
=
\exp\cb{
-c\int_s^t\bb{K*\eta_u}\bb{y}\,du
}
\mu_s(dy).
\end{equation}
\paragraph{Deaths.}
Each particle dies at rate~$\delta$. A death at~$x$ returns resource mass~$m$ at~$x$:
\begin{equation*}
\Lcal_{\mathrm{death}}F\bb{\eta,\mu}
=
\delta
\int_{\R^d}
\Cb{
F\bb{\eta-\delta_x,\mu+m\delta_x}-F\bb{\eta,\mu}
}
\,\eta(dx).
\end{equation*}
Deaths occur only at positions of particles, so~$\eta-\delta_x$ above is a nonnegative discrete measure for~$\eta$-a.e.~$x$. The total death rate is~$\delta\eta\bb{\R^d}$, which vanishes when~$\eta=0$.

The dynamics above define a c\`adl\`ag strong Markov process. We write~$\pP$ and~$\E$ for the process distribution and expectation, and~$(\mathcal F_t)_{t\geq0}$ for its natural filtration. A standard recursive construction is obtained from the birth-and-death construction of~\cite[Sections~2.2--2.3 and Theorem~3.1]{FournierMeleard2004} by adding independent Brownian motions between events. Conditional on the particle trajectories, the resource measure evolves deterministically according to~\eqref{eq:resource-between-deaths}, with mass~$m$ added at each death.

There is a natural nondimensionalization of the model. We measure time in units of the mean particle lifetime~$1/\delta$, length in units of the typical distance a particle moves during its lifetime~$\sigma/\sqrt\delta$, and resource in units of corpse mass~$m$:
\begin{equation}\label{eq:nondimensionalization}
t\mapsto\delta t,
\qquad
x\mapsto\frac{\sqrt{\delta}}{\sigma}x,
\qquad
\eta_t(A)\mapsto\eta_{\frac{t}{\delta}}\bb{\frac{\sigma}{\sqrt{\delta}}A},
\qquad
\mu_t(A)\mapsto\frac1m\mu_{\frac{t}{\delta}}\bb{\frac{\sigma}{\sqrt{\delta}}A}.
\end{equation}
Using the same notation for the nondimensionalized process and kernel, the parameters become
\begin{equation*}
\bb{\beta,\delta,m,c,\sigma,K}
\mapsto
\bb{
\theta,
1,1,1,1,\,
x\mapsto\frac{c}{\delta}K\bb{\frac{\sigma}{\sqrt{\delta}}x}
}, \qquad \theta\defeq\frac{\beta m}{c}.
\end{equation*}
Thus, without loss of generality, particles are standard Brownian motions, have mean lifetime one, and each death returns one unit of resource. From now on the model is determined by~$\theta$ and~$K$.

\section{Main result}\label{sec:main-result}
Let~$\tau_0\defeq\inf\cb{t\geq0: \eta_t\bb{\R^d}=0}$.
We say that the process becomes extinct if~$\tau_0<\infty$ and survives if~$\tau_0=\infty$. Since births require a living particle, extinction is absorbing.

In a mean-field model  every~$m$ units of consumed resource are replenished by one death, so the population grows when the total birth rate exceeds the total consumption rate measured in units of~$m$. By symmetry of~$K$,
\begin{equation*}
\beta\int_{\R^d}\bb{K*\mu}\bb{x}\,\eta(dx)
=
\theta\,\frac{c}{m}
\int_{\R^d}\bb{K*\eta}\bb{y}\,\mu(dy).
\end{equation*}
Thus the total birth rate is~$\theta$ times the total consumption rate measured in units of~$m$, which suggests a naive threshold~$\theta=1$. Our main result shows that this threshold remains sharp in the spatial stochastic model.

\begin{theorem}\label{thm:main}
For every deterministic finite initial state~$\bb{\eta_0,\mu_0}$:
\begin{enumerate}
\item If~$\theta\leq1$, then~$\bP{\tau_0<\infty}=1$.
\item Suppose~$\theta>1$ and, for some~$0<r\leq R$ and~$k_0>0$,
\begin{equation*}
 K\bb{x}=0\quad\text{for }\abs{x}>R,
 \qquad K\bb{x}\geq k_0\quad\text{for }\abs{x}\leq r.
\end{equation*}
Then~$\;\bP{\tau_0=\infty}>0
 \quad\Longleftrightarrow\quad
 \eta_0\bb{\R^d}\geq2\ \text{ or }\ \eta_0\bb{\R^d}=1,\ \mu_0\bb{\R^d}>0$.
\end{enumerate}
\end{theorem}

The condition on the initial state in part~(ii) is necessary: an empty population cannot give birth, and a single particle with no resource has zero birth rate so it dies with no offspring. The nondimensionalization~\eqref{eq:nondimensionalization} leaves~$\theta = \frac{\beta m}{c}$ unchanged and preserves the assumptions on~$K$. Consequently, within the class covered by part~(ii), the extinction--survival characterization depends on the model parameters only through~$\theta$, with sharp threshold~$\theta=1$.

\section{Extinction for \texorpdfstring{$\theta\leq1$}{theta <= 1}}\label{sec:extinction}

The idea of the proof is simple: if~$\theta\leq 1$, then the total mass in the system is a nonnegative supermartingale, and thus converges almost surely. A possible limit cannot have any particles, so the system dies.

\begin{proof}[Proof of Theorem~\ref{thm:main}(i)]
Let~$V_t
\defeq
\eta_t\bb{\R^d}+\theta\mu_t\bb{\R^d}$,
and let~$N_t^b$ and~$N_t^d$ denote the numbers of births and deaths by time~$t$. Let~$A_t$ be the total consumption rate at time~$t$, given by~$A_t
\defeq \int_{\R^d}\bb{K*\eta_t}\bb{y}\,\mu_t(dy)$.
 Between jumps, resource consumption decreases~$V_t$ at rate~$\theta A_t$. A birth increases~$V_t$ by~$1$, while a death removes one particle and adds one unit of resource, so it changes~$V_t$ by~$\theta-1$. Hence,
\begin{equation}\label{eq:V-decomposition}
V_t
=
V_0+N_t^b+(\theta-1)N_t^d
-\theta\int_0^t A_s\,ds.
\end{equation}
By Fubini's theorem and the symmetry of~$K$ we have~$A_t = \int_{\R^d}\bb{K*\mu_t}\bb{x}\,\eta_t(dx)$, so the process~$\theta A_t$ is also a compensator for~$N_t^b$. Thus
\begin{equation*}
M_t^b
\defeq
N_t^b-\theta\int_0^t A_s\,ds,
\qquad
M_t^d
\defeq
N_t^d-\int_0^t\eta_s\bb{\R^d}\,ds
\end{equation*}
are local martingales. Substituting these in~\eqref{eq:V-decomposition} gives
\begin{align*}
V_t
&=
V_0+M_t^b+(\theta-1)M_t^d
+(\theta-1)\int_0^t\eta_s\bb{\R^d}\,ds,
\end{align*}
so if~$\theta\leq1$, then~$V$ is a nonnegative local supermartingale. Consequently, it converges.

Every birth produces a jump of size~$1$ in~$V$, so convergence implies that only finitely many births occur. Hence only finitely many particles ever exist, and since all eventually die, the population becomes extinct in finite time almost surely.
\end{proof}

\section{Survival for \texorpdfstring{$\theta>1$}{theta > 1}}\label{sec:survival}

The proof has three steps. First, we show that if a colony consumes its corpses without much waste, then over a short interval the colony grows by a fixed factor, except with probability exponentially small in the colony size. The point is to turn consumed corpse mass into a lower bound on the number of births. This is the main technical step of the proof. Second, we show that a sufficiently dense colony consumes its corpses efficiently, which allows us to iterate the short-time growth from step one. Over the iteration, the colony size grows exponentially, while the number of fixed-radius balls needed to cover the colony grows only polynomially. Thus, after finitely many steps, some such ball contains twice as many particles as the initial one. Finally, we show that every admissible initial state can reach such a dense ball with positive probability. Restarting the doubling estimate then gives survival with positive probability.

Before we proceed with the proof, we introduce an auxiliary labeling of particles. Let~$\cb{\mathsf T_t}_{t\geq0}$ denote the genealogical forest of particles born by time~$t$. Each birth adjoins a new child to its parent, so the forest grows into~$\mathsf T_\infty$. Let~$\mathsf L_t\subseteq\mathsf T_t$ and~$\mathsf D_t = \mathsf T_t\setminus\mathsf L_t$ be, respectively, the living particles and corpses at time~$t$. For each particle~$\mathsf p\in\mathsf T_t$, let~$X_t\bb{\mathsf p}$ denote its position, fixed after death. Thus~$\eta_t=\sum_{\mathsf p\in\mathsf L_t}\delta_{X_t\bb{\mathsf p}}$. For each particle~$\mathsf p\in\mathsf T_\infty$, let~$b_{\mathsf p}\defeq\inf\cb{t\geq0:\mathsf p\in\mathsf T_t}$ and~$d_{\mathsf p}\defeq\inf\cb{t\geq0:\mathsf p\in\mathsf D_t}$ be its birth and death times.
For~$\mathsf p\in\mathsf D_t$, the mass remaining in its corpse at time~$t$ is
\begin{equation}\label{eq:corpse-mass-evolution}
m_{\mathsf{p}}(t)
\defeq
\exp\cb{
-\int_{d_{\mathsf{p}}}^t
\bb{K*\eta_s}\bb{X_{d_{\mathsf{p}}}\bb{\mathsf{p}}}\,ds
},
\qquad
t\geq d_{\mathsf{p}}.
\end{equation}
We can thus represent evolution of the resource by
\begin{equation}\label{eq:resource-evolution}
\mu_t(dy)
=
\exp\cb{
-\int_0^t\bb{K*\eta_s}\bb{y}\,ds
}\mu_0(dy)
+
\sum_{\mathsf{p}\in\mathsf D_t}
m_{\mathsf{p}}(t)\delta_{X_{d_{\mathsf{p}}}\bb{\mathsf{p}}}(dy).
\end{equation}

\subsection{Consumption implies growth}

The first step of the proof is the bridge between recycling of corpses and new births. We show that, over a short time interval, a large colony is very unlikely to consume most of its corpses without growing. Fix~$\alpha\in\bb{\frac{1}{\theta},1}$. We say that a corpse~$\mathsf{q}$ is \emph{$\alpha$-consumed by time~$t$} if~$m_{\mathsf{q}}(t)\leq1-\alpha$.
This means that at least an~$\alpha$-fraction of the corpse~$\mathsf q$ has been consumed by time~$t$. We choose~$\alpha\theta>1$ so that later we can argue that an~$\alpha$-consumed corpse facilitates, on average, more than one new birth. Choose~$\kappa\in(0,1)$ such that~$(1-\kappa)\alpha\theta>1$. A corpse that dies before the final fraction~$\kappa$ of a short time interval then has enough time left to be~$\alpha$-consumed.

\begin{lemma}\label{lem:short-time-growth-new}
There exist~$h>0$,~$\lambda>1$,~$\gamma>0$, and~$M_*\in\N$ such that the following holds for every~$L\geq 0$. Let~$\tau$ be a finite stopping time, let~$z$ be~$\mathcal F_\tau$-measurable, and let~$\mathsf P\subseteq\mathsf L_\tau$ be an~$\mathcal F_\tau$-measurable collection of~$M\geq M_*$ particles contained in a ball~$B_L\bb{z}$. Let
\begin{equation*}
\mathsf C
\defeq
\cb{
\mathsf{q}\in \mathsf P \cap \mathsf{D}_{\tau + \bb{1-\kappa}h}:
\sup_{\tau\leq s\leq d_{\mathsf{q}}}
\abs{X_s\bb{\mathsf{q}}-X_\tau\bb{\mathsf{q}}}
\leq\frac r4
}
\end{equation*}
be the particles which die sufficiently early and remain nearby. Then
\begin{equation*}
\CP{
\eta_{\tau+h}\bb{B_{L+R+r/2}\bb{z}}<\lambda M,\text{every }\mathsf{q}\in\mathsf C\text{ is $\alpha$-consumed by }\tau+h
}{
\mathcal F_\tau
}
\leq e^{-\gamma M}.
\end{equation*}
\end{lemma}
The set~$\mathsf C$ represents the corpses that the colony can recycle within the short time window: those that die not too close to the end of the window and do not stray too far before dying. The statement is that~$\alpha$-consuming all corpses in~$\mathsf C$ suffices for the population to grow.
\begin{proof}
The argument relies on there being many particles that can make use of the~$\alpha$-consumed corpses from~$\mathsf C$. Let
\begin{equation*}
\mathsf S
\defeq
\cb{
\mathsf{p}\in \mathsf P\cap \mathsf{L}_{\tau+h}:
\sup_{\tau\leq s\leq\tau+h}
\abs{X_s\bb{\mathsf{p}}-X_\tau\bb{\mathsf{p}}}
\leq\frac r4
}.
\end{equation*}
The set~$\mathsf S$ consists of the particles which survive until~$\tau+h$ and do not stray too far. 

The argument requires~$\abs{\mathsf S}$ and~$\abs{\mathsf C}$ to be comparable to~$M$. To show this, let~$W$ be a standard Brownian motion started at the origin, and let~$E\sim\operatorname{Exp}\bb{1}$, independent of each other. The process~$W$ represents the displacement of one particle after time~$\tau$, while~$E$ represents its lifetime. Set
\begin{align*}
s_h
&\defeq
\bP{
E>h,\sup_{0\leq s\leq h}\abs{W_s}\leq\frac r4
}=1-h + o(h),
\\
c_h
&\defeq
\bP{
E\leq(1-\kappa)h,\sup_{0\leq s\leq E}\abs{W_s}\leq\frac r4
}=(1-\kappa)h + o(h).
\end{align*}
Thus~$s_h$ is the probability that a particle belongs to~$\mathsf S$, while~$c_h$ is the probability that it belongs to~$\mathsf C$. The equalities above hold by the reflection principle, which gives~$\bP{\sup_{0\leq s\leq h} \abs{W_s}>\frac{r}{4}}=o(h)$.
Conditional on~$\mathcal F_\tau$, the trajectories and lifespans of the particles in~$\mathsf P$ are independent, with the laws of~$W$ and~$E$. Hence, for every~$\varepsilon<\min\cb{s_h,c_h}$, a Chernoff bound gives, for some~$\gamma'>0$,
\begin{equation}\label{eq:SC-lower-bounds}
\CP{
\abs{\mathsf S}<(s_h-\varepsilon)M
\text{ or }
\abs{\mathsf C}<(c_h-\varepsilon)M
}{
\mathcal F_\tau
}
\leq
2e^{-\gamma' M}.
\end{equation}

Next we show that if corpses in~$\mathsf C$ are~$\alpha$-consumed then there are many new births. For~$s\in[\tau,\tau+h]$, let
\begin{equation*}
\mathsf N_s
\defeq
\cb{
\mathsf{p}\in\mathsf T_s\setminus \mathsf T_{\tau}:
X_{b_{\mathsf{p}}}\bb{\mathsf{p}}\in B_{L+R+r/4}\bb{z}
}
\end{equation*}
be the new particles born in~$B_{L+R+r/4}\bb{z}$ by time~$s$, and set
\begin{equation*}
\Lambda_s
\defeq
\theta\int_\tau^s\int_{B_{L+R+r/4}\bb{z}}
\bb{K*\mu_u}\bb{x}\,\eta_u(dx)\,du.
\end{equation*}
Since~$z$ is~$\mathcal F_\tau$-measurable, the ball~$B_{L+R+r/4}\bb{z}$ is fixed conditional on~$\mathcal F_\tau$. A particle at~$x$ gives birth at rate~$\theta\bb{K*\mu_t}(x)$, so~$\Lambda$ is a compensator of~$\abs{\mathsf N}$. Hence~$\abs{\mathsf N}-\Lambda$ is a local martingale. We obtain the following lower bound on~$\Lambda_{\tau+h}$:
\begin{align*}
\Lambda_{\tau+h}
&=
\theta\int_\tau^{\tau+h}
\int_{B_{L+R+r/4}\bb{z}}
\int_{\R^d}
K\bb{x-y}\,\mu_s(dy)\,\eta_s(dx)\,ds
\\
&\geq
\theta\int_\tau^{\tau+h}
\int_{B_{L+R+r/4}\bb{z}}
\sum_{\mathsf q\in\mathsf D_s}
m_{\mathsf{q}}(s)
K\bb{x-X_{d_{\mathsf{q}}}\bb{\mathsf{q}}}
\,\eta_s(dx)\,ds
\\
&\geq
\theta\int_\tau^{\tau+h}
\int_{B_{L+R+r/4}\bb{z}}
\sum_{\mathsf q\in\mathsf C}
\ind\bb{d_{\mathsf q}\leq s}m_{\mathsf q}(s)
K\bb{x-X_{d_{\mathsf q}}(\mathsf q)}
\,\eta_s(dx)\,ds
\\
&=
\theta\sum_{\mathsf q\in\mathsf C}
\int_{d_{\mathsf q}}^{\tau+h} 
m_{\mathsf q}(s)\int_{B_{L+R+r/4}\bb{z}}
K\bb{X_{d_{\mathsf q}}\bb{\mathsf q}-x}\eta_s(dx)
\,ds
\\
&=
\theta\sum_{\mathsf q\in\mathsf C}
\int_{d_{\mathsf q}}^{\tau+h}
m_{\mathsf q}(s)
\bb{K*\eta_s}\bb{X_{d_{\mathsf q}}(\mathsf q)}
\,ds
\\
&=
\theta
\sum_{\mathsf q\in\mathsf C}
\int_{d_{\mathsf q}}^{\tau+h}
\bb{-\frac{d}{ds}m_{\mathsf q}(s)}\,ds
=
\theta
\sum_{\mathsf q\in\mathsf C}
\bb{1-m_{\mathsf q}(\tau+h)}.
\end{align*}
The first line is the definition of~$\Lambda_{\tau+h}$ with the convolution expanded.
The second line comes from the resource evolution formula~\eqref{eq:resource-evolution} dropping the contribution from~$\mu_0$. The third line is just~$\mathsf{C}\cap\mathsf D_s\subseteq \mathsf D_s$. The fourth line follows from Fubini's theorem and the symmetry of~$K$. The fifth line holds because for~$\mathsf{q}\in\mathsf C$ we have~$X_{d_{\mathsf{q}}}\bb{\mathsf{q}}\in B_{L+r/4}\bb{z}$, and~$K$ has range~$R$, so~$K\bb{X_{d_{\mathsf q}}(\mathsf q)-x}$ vanishes outside~$B_{L+R+r/4}\bb{z}$. The last two identities follow from~\eqref{eq:corpse-mass-evolution}: differentiating gives~$\frac{d}{ds}m_{\mathsf q}(s)
=
-m_{\mathsf q}(s)
\bb{K*\eta_s}\bb{X_{d_{\mathsf q}}(\mathsf q)}$,
and~$m_{\mathsf q}(d_{\mathsf q})=1$.
Therefore, on the event that every corpse in~$\mathsf C$ is~$\alpha$-consumed by~$\tau+h$,
\begin{equation*}
\Lambda_{\tau+h}
\geq \theta\sum_{\mathsf q\in\mathsf C}\bb{1-(1-\alpha)}\geq  
\alpha\theta\abs{\mathsf C}.
\end{equation*}
Consequently,
\begin{equation}\label{eq:compensator-lower-bound-event}
\cb{
\abs{\mathsf C}\geq(c_h-\varepsilon)M
}
\cap
\cb{
\text{every }\mathsf{q}\in\mathsf C
\text{ is $\alpha$-consumed by }\tau+h
}
\subseteq
\cb{
\Lambda_{\tau+h}
\geq
\alpha\theta(c_h-\varepsilon)M
}.
\end{equation}

A lower bound on~$\Lambda$ is useful because it is unlikely that a compensator of a counting process is large and the process small. For~$a>0$, set
\begin{equation*}
Z_s
\defeq
\exp\cb{
-a\abs{\mathsf N_s}
+
\bb{1-e^{-a}}\Lambda_s
},
\qquad
s\in[\tau,\tau+h].
\end{equation*}
Since~$\abs{\mathsf N}-\Lambda$ is a local martingale, we have
\begin{equation*}
dZ_s
=
Z_{s-}\bb{e^{-a}-1}
\,d\bb{\abs{\mathsf N_s}-\Lambda_s}.
\end{equation*}
Thus~$Z$ is a nonnegative local martingale, hence a supermartingale, with~$Z_\tau=1$. Choose~$a>0$ sufficiently small that~$\gamma''\defeq1-e^{-a}-a(1-\varepsilon)>0$. On the event where~$\abs{\mathsf N_{\tau+h}}\leq(1-\varepsilon)v$ and~$\Lambda_{\tau+h}\geq v$, we have~$Z_{\tau+h}\geq e^{\gamma''v}$, so
\begin{equation}\label{eq:compensator-large-process-small-unlikely}
\CP{
\abs{\mathsf N_{\tau+h}}\leq(1-\varepsilon)v,\
\Lambda_{\tau+h}\geq v
}{
\mathcal F_\tau
}
\leq
e^{-\gamma''v}
\CE{Z_{\tau+h}}{\mathcal F_\tau}
\leq
e^{-\gamma''v}.
\end{equation}
Thus~\eqref{eq:compensator-lower-bound-event} and~\eqref{eq:compensator-large-process-small-unlikely} with~$v=\alpha\theta(c_h-\varepsilon)M$ give, possibly after decreasing~$\gamma''$,
\begin{equation}\label{eq:N-lower-bound-very-likely}
\CP{
\begin{gathered}
\abs{\mathsf N_{\tau+h}}
<
(1-\varepsilon)\alpha\theta(c_h-\varepsilon)M,\quad
\abs{\mathsf C}\geq(c_h-\varepsilon)M,
\\
\text{every }\mathsf{q}\in\mathsf C
\text{ is $\alpha$-consumed by }\tau+h
\end{gathered}
}{
\mathcal F_\tau
}
\leq
e^{-\gamma'' M}.
\end{equation}

On the event that~$\abs{\mathsf N_{\tau+h}}$ is large we have many newborns. We verify that not too many of them die or stray too far away. Define the set of newborns in~$\mathsf{N}_{\tau+h}$ that survive until~$\tau+h$ and stay nearby:
\begin{equation*}
\mathsf{NS}
\defeq
\cb{
\mathsf{p}\in\mathsf N_{\tau+h}\cap \mathsf{L}_{\tau+h}:
\sup_{b_{\mathsf{p}}\leq s\leq\tau+h}
\abs{X_s\bb{\mathsf{p}}-X_{b_{\mathsf{p}}}\bb{\mathsf{p}}}
\leq\frac r4
}
\end{equation*}
Set~\mbox{$n\defeq\lfloor(1-\varepsilon)\alpha\theta(c_h-\varepsilon)M\rfloor$}.
Each newborn, conditionally on the history up to its birth time, survives for the next~$h$ units of time and stays within~$r/4$ of its birth location with probability~$s_h$, independently of the past. Such a newborn necessarily belongs to~$\mathsf{NS}$. Thus, on~$\cb{\abs{\mathsf N_{\tau+h}}\geq n}$, the lower tail of~$\abs{\mathsf{NS}}$ is bounded by that of a~$\operatorname{Bin}\bb{n,s_h}$ random variable. Hence a Chernoff bound gives
\begin{equation}\label{eq:NS-lower-bound-very-likely}
\CP{
\abs{\mathsf N_{\tau+h}}\geq n,\
\abs{\mathsf{NS}}<(s_h-\varepsilon)n
}{
\mathcal F_\tau
}
\leq
e^{-\gamma''' M}
\end{equation}
for some~$\gamma'''>0$. Thus it is unlikely that there are many newborns but few survivors.

We check how far the particles from~$\mathsf S$ and~$\mathsf{NS}$ could have gone. Every particle in~$\mathsf S$ lies in~$B_{L+r/4}\bb{z}$. Every particle in~$\mathsf{NS}$ was born in~$B_{L+R+r/4}\bb{z}$ and moves by at most~$r/4$ after birth, so at time~$\tau+h$ it lies in~$B_{L+R+r/2}\bb{z}$. Thus all particles of~$\mathsf S\cup\mathsf{NS}$ lie in~$B_{L+R+r/2}\bb{z}$ at time~$\tau+h$.

The sets~$\mathsf S$ and~$\mathsf{NS}$ are disjoint, because particles in the former existed before time~$\tau$ and those in the latter were born during~$(\tau,\tau+h]$. Therefore
\begin{equation*}
\eta_{\tau+h}\bb{B_{L+R+r/2}\bb{z}}
\geq
\abs{\mathsf S}+\abs{\mathsf{NS}}.
\end{equation*}

Now we choose the parameters. For small~$h$ we have
\begin{equation*}
s_h+s_h\alpha\theta c_h
=
1+\bb{(1-\kappa)\alpha\theta-1}h+o(h)>1.
\end{equation*}
Hence we can choose~$\varepsilon>0$ and then~$\lambda>1$ with
\begin{equation*}
\lambda
<
(s_h-\varepsilon)
\Cb{1+(1-\varepsilon)\alpha\theta(c_h-\varepsilon)},
\end{equation*}
since the right-hand side tends to~$s_h+s_h\alpha\theta c_h$ as~$\varepsilon\to0$. By the definition of~$n$, for all sufficiently large~$M$ this implies
\begin{equation}\label{eq:deterministic-one-step-growth}
(s_h-\varepsilon)M+(s_h-\varepsilon)n
\geq
\lambda M.
\end{equation}

We finish the proof by combining the estimates. By~\eqref{eq:SC-lower-bounds}, we have~$\abs{\mathsf S}\geq(s_h-\varepsilon)M$ and~$\abs{\mathsf C}\geq(c_h-\varepsilon)M$, outside of an event with probability~$2e^{-\gamma' M}$. On the event that every corpse in~$\mathsf C$ is~$\alpha$-consumed by~$\tau+h$, the bound~\eqref{eq:N-lower-bound-very-likely} gives~$\abs{\mathsf N_{\tau+h}}\geq n$, outside an event of probability at most~$e^{-\gamma''M}$. On the event~$\abs{\mathsf N_{\tau+h}}\geq n$, the bound~\eqref{eq:NS-lower-bound-very-likely} gives~$\abs{\mathsf{NS}}\geq(s_h-\varepsilon)n$, outside an event of probability at most~$e^{-\gamma'''M}$. Therefore, by~\eqref{eq:deterministic-one-step-growth},
\begin{equation*}
\eta_{\tau+h}\bb{B_{L+R+r/2}\bb{z}}
\geq
\abs{\mathsf S}+\abs{\mathsf{NS}}
\geq
\lambda M
\end{equation*}
whenever all corpses in~$\mathsf C$ are~$\alpha$-consumed, except on an event with probability no larger than~$2e^{-\gamma' M}+e^{-\gamma'' M}+e^{-\gamma''' M}$. This proves the lemma for~$0<\gamma<\min\cb{\gamma',\gamma'',\gamma'''}$ and sufficiently large~$M_*$.
\end{proof}

\subsection{Doubling of a dense colony}

We now show that an~$r$-ball containing~$H$ particles produces, after a fixed time, an~$r$-ball containing~$2H$ particles except with probability exponentially small in~$H$. Lemma~\ref{lem:short-time-growth-new} shows that a large colony grows over a short interval if every corpse in~$\mathsf C$ is~$\alpha$-consumed. This is precisely what high density provides. A corpse in~$\mathsf C$ has at least~$\kappa h$ time to be consumed, so for the consumption to become very likely, it suffices that there are~$O\bb{\frac{1}{\kappa h}}$ particles nearby. If~$H$ is large, then this requires only a small fraction of the initial population. We reserve part of it to maintain this density throughout the region that the growing colony can reach, and apply Lemma~\ref{lem:short-time-growth-new} repeatedly to the remaining particles. Their number then grows exponentially, while the region they occupy can be covered with polynomially many~$r$-balls. Eventually, one such ball must contain~$2H$ particles.

\begin{proposition}\label{prop:dense-doubling-new}
There are~$T,\gamma>0$ and~$H_0\in\N$, depending only on~$\theta$ and~$K$, such that the following holds. Let~$S$ be a finite stopping time and~$z$ an~$\mathcal F_S$-measurable point. For every~$H\geq H_0$, on~$\cb{\eta_S\bb{B_r\bb{z}}\geq H}$ there is an~$\mathcal F_{S+T}$-measurable point~$z'\in\R^d$ such that
\begin{equation*}
\CP{
\eta_{S+T}\bb{B_r\bb{z'}}\geq2H
}{
\mathcal F_S
}
\geq1-e^{-\gamma H}.
\end{equation*}
\end{proposition}

\begin{proof}
Fix~$k\in\N$, to be chosen later, and set
\begin{equation*}
T\defeq(k+1)h,
\qquad
B_j
\defeq
B_{r+r/4+j(R+r/2)}\bb{z},
\qquad
j=0,\ldots,k,
\end{equation*}
where~$h$ is as in Lemma~\ref{lem:short-time-growth-new}. Write~$B\defeq B_k$. Our goal is to show that, except with probability exponentially small in~$H$, we have~$\eta_{S+T}\bb{B}\gtrsim\lambda^kH$. The radius of~$B$ grows linearly in~$k$, so~$B$ can be covered by~$O\bb{k^d}$ balls of radius~$r$. Since~$\lambda>1$, we can choose~$k$ so that one of these~$r$-balls contains at least~$2H$ particles.

We prove the lower bound on~$\eta_{S+T}\bb{B}$ by repeatedly applying Lemma~\ref{lem:short-time-growth-new} to~$k$ successive time intervals of length~$h$. The~$j$th application starts with a dense colony in~$B_j$ and implies the existence of a dense colony in~$B_{j+1}$. Thus all relevant particles and corpses remain in~$B$. To iterate, we must ensure that every corpse in the corresponding sets~$\mathsf C$ is~$\alpha$-consumed. 

Recall that~$K\bb{x}\geq k_0$ whenever~$\abs{x}\leq r$. Consider one of the growth intervals~$\Cb{t,t+h}$ and a corpse~$\mathsf{q}\in\mathsf C$, as in Lemma~\ref{lem:short-time-growth-new}. Since~$d_{\mathsf q}\leq t+(1-\kappa)h$, the corpse has at least~$\kappa h$ time to be consumed before the end of the interval. If at least~$n$ particles remain within distance~$r$ of~$X_{d_{\mathsf{q}}}\bb{\mathsf{q}}$ throughout~$\Cb{t,t+h}$, then
\begin{equation}\label{eq:many-particles-means-fast-consumption}
m_{\mathsf{q}}(t+h)
\leq
e^{-k_0n(t+h-d_{\mathsf{q}})}
\leq
e^{-k_0n\kappa h}.
\end{equation}
Choose~$n\in\N$ large enough that~$e^{-k_0n\kappa h}\leq1-\alpha$. Thus~$n$ nearby particles are enough to make a corpse~$\alpha$-consumed within time~$\kappa h$.

To apply Lemma~\ref{lem:short-time-growth-new} repeatedly, we only need to verify that there are at least~$n$ particles within distance~$r$ of any point in~$B$, throughout~$\Cb{S+h,S+T}$.  Choose~$\mathcal F_S$-measurable points~$z_1,\ldots,z_l$ such that
\begin{equation*}
B
\subseteq
\bigcup_{m=1}^l B_{r/2}\bb{z_m},
\qquad
l\leq C(1+k)^d,
\end{equation*}
where~$C<\infty$ depends only on~$d,r$, and~$R$. For~$m\leq l$, define
\begin{equation*}
\mathsf D_m
\defeq
\cb{
\begin{aligned}
\mathsf{p}\in\mathsf L_S\cap \mathsf L_{S+T}:\quad&
X_S\bb{\mathsf{p}}\in B_r\bb{z},\\
&
X_s\bb{\mathsf{p}}\in B_{r/2}\bb{z_m}
\text{ for every }s\in[S+h,S+T]
\end{aligned}
}.
\end{equation*}
Thus~$\mathsf D_m$ consists of particles from the initial colony which enter the~$m$th covering ball by time~$S+h$, then remain there and alive until~$S+T$. Let~$\Dcal\defeq\bigcap_{m=1}^l\cb{\abs{\mathsf D_m}\geq n}$. On~$\Dcal$, every point of~$B$ has at least~$n$ particles within distance~$r$ throughout~$\Cb{S+h,S+T}$.

For fixed~$k$, there is~$p>0$ such that, conditional on~$\mathcal F_S$, every particle in~$B_r\bb{z}$ belongs to~$\mathsf D_m$ with probability at least~$p$, uniformly in~$m\leq l$. Indeed, the centers~$z_m$ lie at a bounded distance from~$z$, and a Brownian particle has uniformly positive probability to be in~$B_{r/2}\bb{z_m}$ at time~$S+h$, remain there until~$S+T$, and survive throughout this interval. Conditional on~$\mathcal F_S$, these events are independent for the particles in~$B_r\bb{z}$. Hence, on~$\cb{\eta_S\bb{B_r\bb{z}}\geq H}$, each~$\abs{\mathsf D_m}$ stochastically dominates~$\operatorname{Bin}\bb{H,p}$. Since~$n$ is fixed, a binomial lower-tail bound and a union bound give~$\gamma'>0$ such that, for all~$H$ sufficiently large,
\begin{equation}\label{eq:density-background-failure}
\CP{
\Dcal^c
}{
\mathcal F_S
}
\leq
e^{-\gamma'H}.
\end{equation}

We now set up the~$k$ applications of Lemma~\ref{lem:short-time-growth-new}. Set
\begin{equation*}
t_j\defeq S+(j+1)h,
\qquad
M_0\defeq\left\lfloor\frac{s_hH}{2}\right\rfloor,
\qquad
M_{j+1}\defeq\lfloor\lambda M_j\rfloor,
\end{equation*}
for~$j\leq k$ in the first definition and~$j\leq k-1$ in the last. Thus~$t_0=S+h$,~$t_k=S+T$, and~$M_j$ is a lower bound on population size after~$j$ growth steps, with~$M_k=\frac{s_h}{2}\lambda^kH+O(1)$. If~$\eta_{t_j}\bb{B_j}\geq M_j$ then
let~$\mathsf P_j$ be an~$\mathcal F_{t_j}$-measurable set of~$M_j$ particles in~$B_j$; otherwise let~$\mathsf P_j=\varnothing$. 
Let~$\mathsf C_j$ be the set~$\mathsf C$ from Lemma~\ref{lem:short-time-growth-new} corresponding to~$\mathsf P_j$ on the interval~$\Cb{t_j,t_{j+1}}$,
\begin{equation*}
\mathsf C_j
\defeq
\cb{
\mathsf{q}\in\mathsf P_j\cap \mathsf{D}_{t_j + \bb{1-\kappa}h}:
\sup_{t_j\leq s\leq d_{\mathsf{q}}}
\abs{X_s\bb{\mathsf{q}}-X_{t_j}\bb{\mathsf{q}}}
\leq\frac r4
}.
\end{equation*}
These are precisely the corpses whose consumption is required in the~$j$th growth step.

We claim that on~$\Dcal$ every corpse in each~$\mathsf{C}_j$ is~$\alpha$-consumed by time~$t_{j+1}$. Indeed, if~$\mathsf{q}\in\mathsf C_j$, then its death location belongs to~$B_{j+1}\subseteq B$, so there are at least~$n$ particles within distance~$r$ of the corpse throughout~$\Cb{d_{\mathsf{q}},t_{j+1}}$. By~\eqref{eq:many-particles-means-fast-consumption} and the choice of~$n$, the corpse~$\mathsf{q}$ is~$\alpha$-consumed by time~$t_{j+1}$.

Now we start the iteration. Conditional on~$\mathcal F_S$, each particle in~$B_r\bb{z}$ survives until~$S+h$ and moves by at most~$r/4$ during~$\Cb{S,S+h}$ with probability~$s_h$, independently of the other particles. Every such particle belongs to~$B_0$ at time~$t_0$. Hence, on~$\cb{\eta_S\bb{B_r\bb{z}}\geq H}$,~$\eta_{t_0}\bb{B_0}$ stochastically dominates~$\operatorname{Bin}\bb{H,s_h}$. Thus a Chernoff bound gives~$\gamma''>0$ such that
\begin{equation}\label{eq:initial-population-failure}
\CP{
\eta_{t_0}\bb{B_0}<M_0
}{
\mathcal F_S
}
\leq
\CP{
\eta_{t_0}\bb{B_0}<\frac{s_h H}{2}
}{
\mathcal F_S
}
\leq
e^{-\gamma''H}.
\end{equation}

For~$H$ sufficiently large,~$M_j\geq M_*$ for every~$j\leq k$. Whenever~$\eta_{t_j}\bb{B_j}\geq M_j$, Lemma~\ref{lem:short-time-growth-new}, applied to~$\mathsf P_j$, gives
\begin{equation*}
\CP{
\eta_{t_{j+1}}(B_{j+1})<\lambda M_j,\
\text{every }\mathsf{q}\in\mathsf C_j
\text{ is $\alpha$-consumed by }t_{j+1}
}{
\mathcal F_{t_j}
}
\leq
e^{-\gamma''' M_j},
\end{equation*}
for some~$\gamma'''>0$. To make the iteration explicit, let~$\mathcal A_j\defeq\bigcap_{i=0}^j\cb{\eta_{t_i}\bb{B_i}\geq M_i}$. On~$\Dcal$, every corpse in each~$\mathsf C_j$ is~$\alpha$-consumed by~$t_{j+1}$ and on~$\mathcal{A}_j$ the ball~$B_j$ contains at least~$M_j$ particles. Hence, 
\begin{align*}
\CP{
\Dcal\cap \mathcal A_j\cap \mathcal A_{j+1}^c
}{
\mathcal F_S
}&=
\CE{
\ind_{\mathcal A_j}
\CP{
\Dcal\cap 
\cb{\eta_{t_{j+1}}(B_{j+1})<M_{j+1}}
}{
\mathcal F_{t_j}
}
}{
\mathcal F_S
}\\
&\leq
\CE{
\ind_{\mathcal A_j}
e^{-\gamma''' M_j}
}{
\mathcal F_S
} \leq e^{-\gamma''' M_j},
\end{align*}
so the probability that the~$j$th growth step fails is exponentially small in~$H$. Thus,
\begin{align}
& \CP{
\Dcal\cap \mathcal A_0\cap
\cb{\eta_{S+T}\bb{B}<M_k}
}{
\mathcal F_S
}
\leq
\CP{
\Dcal\cap \mathcal A_0\cap
\mathcal{A}_k^c
}{
\mathcal F_S
}
\nonumber\\
& \qquad \leq
\sum_{j=0}^{k-1}
\CP{
\Dcal\cap \mathcal A_j\cap \mathcal A_{j+1}^c
}{
\mathcal F_S
}
  \leq
\sum_{j=0}^{k-1}e^{-\gamma'''M_j}
\leq
e^{-\gamma''''H}, \label{eq:growth-iteration-failure}
\end{align}
for some~$\gamma''''>0$ and all~$H$ sufficiently large. We have
\begin{equation*}
\cb{\eta_{S+T}\bb{B}<M_k}
\subseteq
\Dcal^c
\cup
\mathcal A_0^c
\cup
\bb{
\Dcal\cap\mathcal A_0\cap
\cb{\eta_{S+T}\bb{B}<M_k}
}.
\end{equation*}
Therefore, by~\eqref{eq:density-background-failure}, \eqref{eq:initial-population-failure}, and~\eqref{eq:growth-iteration-failure},
\begin{equation*}
\CP{
\eta_{S+T}\bb{B}<M_k
}{
\mathcal F_S
}
\leq
e^{-\gamma'H}
+
e^{-\gamma''H}
+
e^{-\gamma''''H}\leq
e^{-\gamma H}
\end{equation*}
for some~$\gamma>0$ and all~$H$ sufficiently large.

It remains to choose~$k$ and~$z'$. Since~$M_{j+1}\geq\lambda M_j-1$, we have~$M_k
\geq
\lambda^kM_0-\frac{\lambda^k}{\lambda-1}
\geq
\frac{s_h}4\lambda^kH$
for all~$H$ sufficiently large. The ball~$B$ is covered by the balls~$B_{r/2}\bb{z_i}\subseteq B_r\bb{z_i}$,~$i\leq l$, with~$l\leq C(1+k)^d$ and~$C$ depending only on~$d$,~$r$ and~$R$. Since~$\lambda>1$, we can fix~$k$, depending only on the model parameters, such that~$\frac{s_h}{4}\lambda^k
\geq
2C(1+k)^d$.
On~$\cb{\eta_{S+T}\bb{B}\geq M_k}$ some~$z_i$ then satisfies
\begin{equation*}
\eta_{S+T}\bb{B_r\bb{z_i}}
\geq
\frac{M_k}{l}
\geq
\frac{s_h\lambda^kH}{4C(1+k)^d}
\geq
2H.
\end{equation*}
Let~$z'$ be~$z_i$ for the smallest such~$i$, and~$z'\defeq z_1$ if there is none. Then~$z'$ is~$\mathcal F_{S+T}$-measurable and
\begin{equation*}
\CP{
\eta_{S+T}\bb{B_r\bb{z'}}<2H
}{
\mathcal F_S
}
\leq
\CP{
\eta_{S+T}\bb{B}<M_k
}{
\mathcal F_S
}
\leq
e^{-\gamma H}. \qedhere
\end{equation*}
\end{proof}

\subsection{Entrance lemma and proof of Theorem~\ref{thm:main}(ii)}

Proposition~\ref{prop:dense-doubling-new} shows that once the process reaches a sufficiently dense local configuration, the growth argument can be restarted indefinitely with positive probability. It remains only to show that every initial state in Theorem~\ref{thm:main}(ii) can reach such a configuration.

\begin{lemma}\label{lem:entrance-new}
For every~$H\in\N$ and every finite initial state with at least two particles, or one particle and positive resource,
there exists an~$r$-ball~$B$ such that for any~$T>0$ we have~$\bP{
\eta_T\bb{B}\geq H
}
>0$.
\end{lemma}

\begin{proof}
We first treat the case of one particle~$\mathsf p$ and~$\mu_0\bb{\R^d}>0$. Then there exists an~$r/2$-ball~$B'$ such that~$\mu_0\bb{B'}\geq \varepsilon$ for some~$\varepsilon>0$. Let~$B$ be an~$r$-ball with the same center as~$B'$. For any~$x,y\in B'$ we have~$K\bb{x-y}\geq k_0$. Consider the event that~$\mathsf p$ enters~$B'$ by time~$\frac{T}{2}$ without branching, does not leave~$B'$ before~$T$, survives until~$T$, branches exactly~$H-1$ times while in~$B'$, and its offspring never leave~$B'$, do not branch and do not die before~$T$. Because~$K$ is bounded, on this event we have~$\bb{K*\eta_t}\bb{y}\leq H\norm{K}_\infty$, so by~\eqref{eq:resource-between-deaths} the resource available in~$B'$ stays uniformly bounded below throughout~$\Cb{0,T}$: 
\begin{equation*}
\mu_t\bb{B'}
\geq
\varepsilon e^{-H\norm K_\infty t}
\geq
\varepsilon e^{-H\norm K_\infty T}.
\end{equation*}
Thus the rate at which~$\mathsf p$ branches while in~$B'$ is at least~$\theta k_0\varepsilon e^{-H\norm K_\infty T}$, and it is at most~$\theta\norm K_\infty\mu_0\bb{\R^d}$, since on this event there are no deaths in~$\Cb{0,T}$, so~$\mu_t\bb{\R^d}\leq \mu_0\bb{\R^d}$. Since the particle motions are independent of the branching and of each other, the event has positive probability. On it,~$\eta_T\bb{B}\geq\eta_T\bb{B'}=H$.

The second case,~$\eta_0\bb{\R^d}\geq2$, reduces to the first with positive probability, via the event that all initial particles but one die before any branching.
\end{proof}
\begin{proof}[Proof of Theorem~\ref{thm:main}(ii)]
Let~$T$,~$\gamma$, and~$H_0$ be as in Proposition~\ref{prop:dense-doubling-new}, and choose~$H\geq H_0$ so large that $\sum_{k=0}^{\infty}e^{-\gamma2^kH}<1$.
By Lemma~\ref{lem:entrance-new}, there is an~$r$-ball~$B_0$ such that~$\bP{\eta_T\bb{B_0}\geq H}>0$.

On~$\cb{\eta_T\bb{B_0}\geq H}$, apply Proposition~\ref{prop:dense-doubling-new} successively at times~$T,2T,\ldots$, using at each step the measurable ball supplied by the proposition. If the first~$k$ doublings have succeeded, the~$k$th application starts from an~$r$-ball containing at least~$2^kH$ particles, so its conditional probability of failure is at most~$e^{-\gamma2^kH}$. Hence, by a union bound,
\begin{equation*}
\CP{
\eta_{(k+1)T}\bb{B_k}\geq2^kH
\text{ for every }k\geq0
}{
\eta_T\bb{B_0}\geq H
}
\geq
1-\sum_{k=0}^{\infty}e^{-\gamma2^kH}
>0,
\end{equation*}
where~$B_k$ denotes the ball selected after~$k$ successful steps. Since the entrance event also has positive probability, with positive probability the population is nonempty at all times~$(k+1)T$. Extinction is absorbing, so on this event~$\tau_0=\infty$. Thus~$\bP{\tau_0=\infty}>0$.
\end{proof}

\begin{acks}
This material is based upon work supported by the NSF Grant No.~DMS-2348113 awarded to Marek Biskup. The authors also thank him for insightful discussions. Generative AI tools (Claude and ChatGPT) were used to assist with editing the text, searching the literature, and checking the proofs for gaps. All mathematical content was verified by the authors, who take full responsibility for the manuscript.
\end{acks}

\end{document}